\documentclass[11pt,reqno]{amsart}

\usepackage{amsmath,amssymb,mathtools}
\usepackage[margin=1.15in]{geometry}
\usepackage[hidelinks]{hyperref}

\hypersetup{
 pdftitle={A High-Rank Gap Theorem for Type-I Bounded Symmetric Domains},
 pdfauthor={Yun Gao},
 pdfkeywords={bounded symmetric domain, Shilov boundary, gap phenomenon,
 Levi form, Hopf boundary lemma}
}

\newcommand{\C}{\mathbb C}
\newcommand{\Span}{\operatorname{span}}
\newcommand{\rankop}{\operatorname{rank}}

\newcommand{\Aut}{\operatorname{Aut}}

\newtheorem{theorem}{Theorem}[section]
\newtheorem{proposition}[theorem]{Proposition}
\newtheorem{lemma}[theorem]{Lemma}
\newtheorem{corollary}[theorem]{Corollary}
\newtheorem*{realift}{Real implicit function theorem}
\theoremstyle{definition}
\newtheorem{definition}[theorem]{Definition}
\theoremstyle{remark}
\newtheorem{remark}[theorem]{Remark}

\title[High-rank gaps for Type-I domains]
{A High-Rank Gap Theorem for Type-I Bounded Symmetric Domains}
\author{Yun Gao}
\thanks{School of Mathematical Sciences, Shanghai Jiao Tong University, Shanghai,
	People's Republic of China. \textbf{Email:}~gaoyunmath@sjtu.edu.cn.
	Supported by the NSFC under Grant Nos.~12471042 and~12471078.}

\date{}

\subjclass[2020]{32H02, 32H35, 32V20}
\keywords{bounded symmetric domain, Shilov boundary, gap phenomenon,
Levi form, Hopf boundary lemma}

\begin{document}

\begin{abstract}
Let $\Omega_{r,s}$ and $\Omega_{r',s'}$ be Type-I bounded symmetric
domains with $s>r$ and $s'>r'$.  Let
$F:U\to M_{r',s'}(\C)$ be holomorphic, where
$\overline{\Omega}_{r,s}\subset U$, and suppose that
$F(S_{r,s})\subset S_{r',s'}$.  If
\[
 k(s-r)\leq s'-r'<(k+1)(s-r)
\quad\text{and}\quad r'>kr,
\]
then, up to target coordinates, the map contains a fixed identity block
of size $r'-kr$.  The proof first turns the boundary equation into a
matrix identity for the first derivative.  A direct dimension count gives
$g(s-r)\leq s'-r'$, where $g$ is the rank of a positive semidefinite matrix arising
from this identity.  The extremal property of the Shilov boundary and
the scalar Hopf boundary lemma then show that certain rows of the map
are constant.  This gives the fixed identity block.  We also give
examples showing that its size is optimal and explain why the proof
needs $F$ on a neighborhood of $\overline{\Omega}_{r,s}$.
\end{abstract}

\maketitle

\section{Introduction}

The rigidity of holomorphic maps between bounded symmetric domains is a
classical theme in several complex variables.  Its rank-one model is the
unit ball $B^n$.  Classical work of Poincar\'e on sphere equivalences
already exhibited the exceptional rigidity of this model \cite{Poincare};
in higher dimension Alexander proved that every proper holomorphic
self-map of $B^n$, $n\geq2$, is an automorphism \cite{Alexander}.  When the target
dimension is allowed to increase, nonlinear proper maps occur and the
central question becomes how much extra dimension is needed before new
maps can appear.  The boundary regularity and algebraicity results of
Cima--Suffridge and Forstneri\v{c} provided an important 
foundation for this problem \cite{CS,F1,F2}, while D'Angelo constructed
and studied basic nonlinear polynomial examples \cite{Da1,Da2}.

This dimension problem led to the \emph{gap phenomenon} for proper maps
between balls.  Faran proved that a rational proper map
$B^n\to B^N$ is equivalent to a linear embedding when
$n+1\leq N\leq2n-2$ \cite{Fa}; Huang subsequently obtained the
corresponding low-codimension linearity under local boundary regularity
hypotheses \cite{Hu1}.  At the first critical dimension $N=2n-1$,
Huang--Ji obtained a precise classification \cite{HJ}.  Later work of
Huang--Ji--Xu and Huang--Ji--Yin revealed further forbidden intervals
in the target dimension \cite{hjx,HJY}.  These results led to the
general Gap Conjecture, which predicts a sequence of dimension ranges
on which every proper ball map has smaller geometric rank and hence
reduces to a lower-dimensional target; see \cite{HJY2} for an overview.

For bounded symmetric domains of higher rank, rigidity has been studied
by both geometric and analytic methods; see, for example,
\cite{MT,Tu1,Tu2,MNT,Ng1,Ng2,Chan1,Chan2}.  For Type-I domains, Kim
studied holomorphic maps between closed Grassmannian orbits
\cite{Kim1}, while Kim--Zaitsev proved rigidity theorems for local CR
maps between Shilov boundaries and for proper holomorphic maps
\cite{KZ13,KZ15}.  These results suggest that gap phenomena should not
be confined to maps between balls.  The mixed-rank problem, in which the
source has rank one and the target has higher rank, was treated in
\cite{GaoGap}.  The purpose of the present paper is to establish the
corresponding gap intervals when the source also has higher rank.

For integers $s>r\geq1$, the Type-I bounded symmetric domain is
\[
 \Omega_{r,s}
 :=\{Z\in M_{r,s}(\C): I_r-ZZ^H>0\}.
\]
Its Shilov boundary is
\[
 S_{r,s}
 :=\{Z\in M_{r,s}(\C):ZZ^H=I_r\}.
\]
For $Z\in M_{r,s}(\C)$, associate the $r$-dimensional space
\[
 V_Z:=\{(u,uZ):u\in\C^r\}\subset\C^{r+s}
\]
and define
\[
 I_{r,s}:=\begin{pmatrix}I_r&0\\0&-I_s\end{pmatrix},
\]
write
\[
 \langle\xi,\eta\rangle_{r,s}:=\xi I_{r,s}\eta^H
 \qquad(\xi,\eta\in\C^{r+s})
\]
where vectors are written as rows unless otherwise specified, and
$A^H=\overline A^{\,t}$ denotes the Hermitian transpose.  Then the form
is positive on $V_Z$ for $Z\in\Omega_{r,s}$, whereas $V_Z$ is a maximal
null $r$-plane for $Z\in S_{r,s}$.  Moreover,
\[
 V_Z\perp V_W\quad\Longleftrightarrow\quad I_r-ZW^H=0.
\]
Thus the Type-I domain, its Shilov boundary, and the orthogonality
relation are encoded by the same Hermitian form.  The proof in
\cite{GaoGap} is elementary and is based on this orthogonal structure.
In the rank-one case, a CR map between open pieces of Shilov boundaries
extends holomorphically, and polarization shows that the extension
preserves orthogonality.  One can then choose mutually orthogonal source
points and count the dimensions of the corresponding target spaces.
This gives the sharp intervals
\[
 k(s-1)\leq s'-r'<(k+1)(s-1)
\]
and a fixed null subspace of dimension at least $r'-k$.

In the present paper, we consider a holomorphic map satisfying the
global boundary hypothesis
\[
 F:U\longrightarrow M_{r',s'}(\C),\qquad
 \overline{\Omega}_{r,s}\subset U,
 \qquad F(S_{r,s})\subset S_{r',s'},
\]
where $U$ is connected.  The same orthogonal viewpoint remains useful
when the source rank is $r\geq2$, but the boundary equation is now
matrix-valued.  After polarization near a boundary point $Z_0$, we
obtain a family of
complex-linear maps $\mathcal M_{Z,\overline W}:M_r(\C)\to M_{r'}(\C)$
such that
\[
 \mathcal M_{Z,\overline W}(I_r-ZW^H)
 =I_{r'}-F(Z)F(W)^H.
\]
In particular, $F$ preserves orthogonality locally.  Although both
estimates ultimately use orthogonality, the higher-rank argument is not
a direct extension of the rank-one dimension count in \cite{GaoGap}.
In the rank-one case, the estimate is obtained by choosing mutually
orthogonal source points and comparing the dimensions of sums and
intersections of their image spaces.  Here the polarized identity must
first be differentiated along the Shilov boundary.  The values of
$\mathcal M_{Z_0}$ on the matrix units form a positive semidefinite block
matrix, whose factorization gives
\[
 \mathcal M_{Z_0}(T)
 =\sum_{\mu=1}^{g(Z_0)}A_\mu T A_\mu^H.
\]
Separating the $s-r$ source column directions then produces mutually
orthogonal subspaces of $\C^{s'-r'}$, all of dimension $g(Z_0)$, and
hence
\[
 g(Z_0)(s-r)\leq s'-r'.
\]
Proving this matrix factorization and extracting the $s-r$ orthogonal
target subspaces from it constitute one of the main steps of the proof.

A second difficulty is that this calculation initially produces a null
subspace only at $Z_0$ and does not show that the same subspace is
contained in $V_{F(Z)}$ at other points.

The condition on the whole Shilov boundary supplies the additional
global argument.  It first implies
$I_{r'}-F(Z)F(Z)^H\succeq0$ throughout $\Omega_{r,s}$.  Applied to
suitable analytic discs, the Hopf boundary lemma then shows that the
null subspace obtained at $Z_0$ is contained in $V_{F(Z)}$ for every
$Z$.  The fixed subspace gives the block decomposition in the following
theorem.

\begin{theorem}\label{thm:main}
Let $s>r\geq2$ and $s'>r'$.  Suppose that
$F:U\to M_{r',s'}(\C)$ is holomorphic on a connected open neighborhood
$U$ of $\overline{\Omega}_{r,s}$ and
$F(S_{r,s})\subset S_{r',s'}$.
Let $k\geq1$ be the integer determined by
\begin{equation}\label{eq:gap}
 k(s-r)\leq s'-r'<(k+1)(s-r).
\end{equation}
If $r'>kr$, then, after unitary changes of the target row and column
coordinates,
\begin{equation}\label{eq:main-block}
 F(Z)=
 \begin{pmatrix}
  I_{r'-kr}&0\\
  0&F_0(Z)
 \end{pmatrix}.
\end{equation}
Here
\[
 F_0:U\longrightarrow
 M_{kr,\,kr+s'-r'}(\C)
\]
is holomorphic and satisfies
\[
 F_0(S_{r,s})\subset S_{kr,\,kr+s'-r'},
 \qquad
 F_0(\Omega_{r,s})\subset
 \overline{\Omega}_{kr,\,kr+s'-r'}.
\]
\end{theorem}

\begin{remark}\label{rem:sharpness}
The dimension $r'-kr$ of the fixed block is optimal.  Write
\[
 s'-r'=k(s-r)+\ell,
 \qquad 0\leq\ell<s-r,
\]
and set $d=r'-kr$.  Then $r'=d+kr$ and $s'=d+ks+\ell$.  If
\[
 D_k(Z):=\operatorname{diag}(Z,\ldots,Z)\in M_{kr,ks}(\C)
\]
has $k$ copies of $Z$, define
\[
 F_k(Z):=
 \begin{pmatrix}
  I_d&0&0\\
  0&D_k(Z)&0_{kr\times\ell}
 \end{pmatrix}.
\]
Since
$D_k(Z)D_k(Z)^H=\operatorname{diag}(ZZ^H,\ldots,ZZ^H)$, this map sends
$S_{r,s}$ into $S_{r',s'}$ and $\overline{\Omega}_{r,s}$ into
$\overline{\Omega}_{r',s'}$.

The common subspace of the planes $V_{F_k(Z)}$, $Z\in\Omega_{r,s}$, is
exactly the $d$-dimensional null subspace determined by the identity
block.  To see this, write the first $r'$ coordinates of a common vector
as $x=(x_0,x_1,\ldots,x_k)$, where $x_0\in\C^d$ and
$x_j\in\C^r$.  Since $xF_k(Z)$ is independent of $Z$ and
$\Omega_{r,s}$ is open, each $x_jZ$ can be constant only when $x_j=0$
for $1\leq j\leq k$.  Thus the fixed block cannot be enlarged.
\end{remark}

Every space $V_{F(Z)}$ therefore contains the same null
$(r'-kr)$-dimensional subspace.  Since a positive $r'$-plane contains no
nonzero null vector, such an $F$ cannot map $\Omega_{r,s}$ into
$\Omega_{r',s'}$.  The lower inequality in \eqref{eq:gap} identifies the
$k$-th interval, while the proof uses the strict upper inequality.

The proof also works when $r=1$.  In that case its numerical condition
and block decomposition agree with those in \cite{GaoGap}.  The two
results have different hypotheses and use the orthogonal structure in
different ways.  The paper \cite{GaoGap} starts with a CR map defined on
an open piece of the sphere and studies its holomorphic extension as a
map preserving orthogonality between Grassmannians.  Here the map is assumed from the
outset to be holomorphic near the whole closed domain and to preserve the
entire Shilov boundary.  This stronger global hypothesis allows the
fixed subspace to be obtained directly from the first derivative at one
boundary point and the Hopf boundary lemma.  We state
Theorem~\ref{thm:main} for $r\geq2$ because the rank-one case has already
been established in \cite{GaoGap}.  A related use of the Hopf lemma
appears in Kim's work on proper maps at general Shilov points
\cite{KimDuke}.

The proof is organized as follows.  Section~\ref{sec:Levi} introduces
the orthogonal Grassmannian terminology and proves
Lemma~\ref{lem:Levi} directly from the boundary equations.
Section~\ref{sec:positive}
shows that positive points map to semi-positive points.
Section~\ref{sec:linear} proves the dimension estimate, and
Section~\ref{sec:Hopf} constructs the null subspace and proves
that it is independent of $Z$.  Theorem~\ref{thm:main} is proved in
Section~\ref{sec:proof}.

\section{Orthogonal structure and the Levi form}
\label{sec:Levi}

Let $\mathcal G(r,s)$ denote the Grassmannian $G(r,r+s)$ endowed with
the Hermitian form of signature $(r,s)$ defined by $I_{r,s}$.  A point
$\mathcal Z\in\mathcal G(r,s)$ corresponds to an $r$-plane
$V_{\mathcal Z}\subset\C^{r+s}$.  If $A_{\mathcal Z}$ is any
$r\times(r+s)$ matrix whose row space is $V_{\mathcal Z}$, then
$A_{\mathcal Z}$ is called a representative matrix of $\mathcal Z$.

\begin{definition}\label{def:orthogonal-points}
For $\mathcal Z,\mathcal W\in\mathcal G(r,s)$, we say that
$\mathcal Z$ and $\mathcal W$ are \emph{orthogonal}, and write
$\mathcal Z\perp\mathcal W$, if
\[
 A_{\mathcal Z}I_{r,s}A_{\mathcal W}^H=0.
\]
A point $\mathcal Z\in\mathcal G(r,s)$ is called
\emph{positive}, \emph{semi-positive}, or \emph{null} according as
\[
 A_{\mathcal Z}I_{r,s}A_{\mathcal Z}^H>0,
 \qquad
 A_{\mathcal Z}I_{r,s}A_{\mathcal Z}^H\succeq0,
 \qquad\text{or}\qquad
 A_{\mathcal Z}I_{r,s}A_{\mathcal Z}^H=0,
\]
respectively.  A linear subspace $K\subset\C^{r+s}$ is called
\emph{null} if the Hermitian form vanishes identically on
$K\times K$.
\end{definition}

Replacing $A_{\mathcal Z}$ and $A_{\mathcal W}$ by
$XA_{\mathcal Z}$ and $YA_{\mathcal W}$, with $X,Y$ invertible,
changes $A_{\mathcal Z}I_{r,s}A_{\mathcal W}^H$ to
$X(A_{\mathcal Z}I_{r,s}A_{\mathcal W}^H)Y^H$.  Thus its vanishing is
unchanged.  When $\mathcal Z=\mathcal W$, this is a congruence
transformation and hence also preserves positivity and semi-positivity.
Therefore, the notions in Definition~\ref{def:orthogonal-points} do not
depend on the representative matrices.

In the standard affine chart, write
\[
 \mathcal Z=[\,I_r,Z\,],
 \qquad
 V_{\mathcal Z}=\{(u,uZ):u\in\C^r\}.
\]
Then
\[
 [\,I_r,Z\,]I_{r,s}[\,I_r,Z\,]^H=I_r-ZZ^H.
\]
Consequently, the positive, semi-positive, and null points in this
chart are precisely the points represented by matrices in
$\Omega_{r,s}$, $\overline{\Omega}_{r,s}$, and $S_{r,s}$,
respectively.  Moreover,
\[
 \mathcal Z\perp\mathcal W
 \quad\Longleftrightarrow\quad
 [\,I_r,Z\,]I_{r,s}[\,I_r,W\,]^H
 =I_r-ZW^H=0.
\]

The purpose of this section is to obtain the first-derivative identity
that will drive the dimension count.  We normalize one boundary point,
compute the complex tangent space and the Levi form, and polarize the
boundary equation satisfied by $F$.  Differentiating the resulting
factorization gives Lemma~\ref{lem:Levi}.

Choose an arbitrary point $Z_0\in S_{r,s}$. Unitary
changes of source and target column coordinates allow us to assume
\[
 Z_0=(I_r,0),\qquad F(Z_0)=(I_{r'},0).
\]
Write $Z=(A,B)$ with
$A\in M_r(\C)$ and $B\in M_{r,s-r}(\C)$.  A matrix defining function
for $S_{r,s}$ near $Z_0$ is
\[
 \rho(Z,\overline Z):=I_r-AA^H-BB^H.
\]
On the diagonal we abbreviate this expression as $\rho(Z)$.
A $(1,0)$ tangent vector at $Z=(A,B)$ will be written as
$(P,Q)\in M_r(\C)\times M_{r,s-r}(\C)$.  Direct differentiation gives
\[
 (\partial\rho)_Z(P,Q)=-PA^H-QB^H.
\]
At $Z_0=(I_r,0)$ this becomes
$(\partial\rho)_{Z_0}(P,Q)=-P$.  Hence $(P,Q)$ is a complex tangent
vector exactly when $P=0$, and consequently
\[
 T_{Z_0}^{1,0}S_{r,s}
 =\{(0,X):X\in M_{r,s-r}(\C)\}
 \simeq M_{r,s-r}(\C).
\]

For $X,Y\in M_{r,s-r}(\C)$, define the Levi form of $\rho$ by
\[
 \mathcal L_{Z_0}(X,\overline Y)
 :=-(\partial\bar\partial\rho)_{Z_0}
       \bigl((0,X),\overline{(0,Y)}\bigr).
\]
To see the Levi form directly, let $\tau$ and $\sigma$ be complex
scalars and note that the coefficient of
$\tau\overline\sigma$ in
\[
 -(\tau X+\sigma Y)(\tau X+\sigma Y)^H
\]
is $-XY^H$.  The minus sign in the definition of $\mathcal L_{Z_0}$
therefore gives
\begin{equation}\label{eq:source-Levi}
 \mathcal L_{Z_0}(X,\overline Y)=XY^H.
\end{equation}

At $Z_0=(I_r,0)$ and for real $t$,
\[
 \rho(I_r+tP,tQ)
 =-t(P+P^H)-t^2(PP^H+QQ^H).
\]
Consequently, the real differential of $\rho$ is
\[
 d\rho_{Z_0}(P,Q)=-(P+P^H).
\]
For every Hermitian matrix $H$,
\[
 d\rho_{Z_0}(-H/2,0)=H.
\]
Hence $d\rho_{Z_0}$ is surjective onto the Hermitian matrices.

We use the following standard submersion form of the real implicit function
theorem.  Its $C^1$ version follows from \cite[Theorem~9.28]{Rudin} after
choosing a complement to the kernel of the differential, and the same proof
using the $C^k$ inverse function theorem gives the stated regularity.

\begin{realift}
Let $E$ and $F$ be finite-dimensional real vector spaces, let
$\mathcal O\subset E$ be open, and let
$\Phi:\mathcal O\to F$ be a $C^k$ map, where $1\leq k\leq\infty$.
Suppose that $p\in\mathcal O$ and that
$d\Phi_p:E\to F$ is surjective.  Put $K=\ker d\Phi_p$.  Then there are
neighborhoods $\mathcal O_0$ of $p$, $\mathcal U_0$ of $0$ in $K$, and
$\mathcal V_0$ of $0$ in $F$, together with a $C^k$ diffeomorphism
\[
 \Psi:\mathcal O_0\longrightarrow
 \mathcal U_0\times\mathcal V_0,
 \qquad \Psi(p)=(0,0),
\]
such that
\[
 \Phi\bigl(\Psi^{-1}(u,v)\bigr)=\Phi(p)+v
 \qquad
 ((u,v)\in\mathcal U_0\times\mathcal V_0).
\]
Consequently, $\Phi^{-1}(\Phi(p))\cap\mathcal O_0$ is a $C^k$
submanifold of $E$ of codimension $\dim_{\mathbb R}F$, and
\[
 T_p\bigl(\Phi^{-1}(\Phi(p))\bigr)=\ker d\Phi_p.
\]
\end{realift}

Apply this theorem with
\[
 E=M_{r,s}(\C),\qquad
 F=\operatorname{Herm}_r,\qquad
 \Phi=\rho,\qquad p=Z_0,
\]
where the matrix spaces are regarded as real vector spaces and
$\operatorname{Herm}_r$ denotes the real vector space of Hermitian
$r\times r$ matrices.  Since $\rho(Z_0)=0$ and
$d\rho_{Z_0}$ is surjective, it follows that
$S_{r,s}=\rho^{-1}(0)$ is a smooth real submanifold near $Z_0$ of real
codimension $r^2$, and its tangent space is
\[
 T_{Z_0}S_{r,s}=\ker d\rho_{Z_0}.
\]

\begin{proposition}\label{prop:polarized-factorization}
There is a connected open neighborhood $U_0\subset U$ of $Z_0$ and a
family of complex-linear maps
\[
 \mathcal M_{Z,\overline W}:M_r(\C)\longrightarrow M_{r'}(\C),
 \qquad Z,W\in U_0,
\]
depending holomorphically on $(Z,\overline W)$ and satisfying
\begin{equation}\label{eq:polarized-factor}
 \mathcal M_{Z,\overline W}\bigl(I_r-ZW^H\bigr)
 =I_{r'}-F(Z)F(W)^H.
\end{equation}
\end{proposition}

\begin{proof}
For a second matrix variable $W\in M_{r,s}(\C)$, put
\[
 \rho(Z,\overline W):=I_r-ZW^H,
 \qquad
 G(Z,\overline W):=I_{r'}-F(Z)F(W)^H.
\]
Here and below, $Z$ and $\overline W$ are regarded as independent complex
variables.  Thus $\rho(Z,\overline W)$ and $G(Z,\overline W)$ are
holomorphic in the independent variables $(Z,\overline W)$.  The boundary
hypothesis says $G(Z,\overline Z)=0$ whenever
$\rho(Z,\overline Z)=0$.  To polarize this identity, locally parametrize
the real-analytic manifold $S_{r,s}$ by
$2rs-r^2$ real variables and allow these variables to be complex.
The original boundary identity vanishes for all real values of the
parameters, so the holomorphic identity theorem makes it vanish for
their complex values.  The complexified parametrization has dimension
$2rs-r^2$ and lies in the set $ZW^H=I_r$.  The latter is a complex
submanifold of the same dimension near $(Z_0,\overline Z_0)$ because
the differential of $\rho$ has rank $r^2$ there.  Thus the two agree
after shrinking the neighborhood.  Consequently,
\begin{equation}\label{eq:polarized-boundary}
 G(Z,\overline W)=0
 \quad\text{whenever}\quad
 \rho(Z,\overline W)=0,
\end{equation}
for $(Z,\overline W)$ near $(Z_0,\overline Z_0)$.

At $(Z_0,\overline Z_0)$, keep $\overline W$ fixed and vary $Z$ in the
direction $(P,0)$, where $P\in M_r(\C)$.  Then
\[
 D_Z\rho_{(Z_0,\overline Z_0)}(P,0)=-P.
\]
Since $P$ is arbitrary, the differential of $\rho$ is surjective onto
$M_r(\C)$.  We now explain explicitly how the remaining local coordinates
are chosen.  Let
\[
 E:=M_{r,s}(\C)\times M_{r,s}(\C),
 \qquad
 D\rho_0:=D\rho_{(Z_0,\overline Z_0)}:E\longrightarrow M_r(\C).
\]
Then
\[
 \dim_{\C}\ker D\rho_0=2rs-r^2.
\]
Choose a complex subspace $E_1\subset E$ such that
\[
 E=E_1\oplus\ker D\rho_0.
\]
The restriction $D\rho_0|_{E_1}:E_1\to M_r(\C)$ is an isomorphism.
Choose a complex-linear map
\[
 \Theta_0:E\longrightarrow\C^{2rs-r^2}
\]
which vanishes on $E_1$ and whose restriction to $\ker D\rho_0$ is an
isomorphism onto $\C^{2rs-r^2}$.  Such a map is obtained by choosing a
basis of $\ker D\rho_0$.  Put
\[
 \Theta(Z,\overline W)
 :=\Theta_0(Z-Z_0,\overline W-\overline Z_0)
\]
and define
\[
 \Phi(Z,\overline W)
 :=\bigl(\rho(Z,\overline W),\Theta(Z,\overline W)\bigr).
\]
Its differential at $(Z_0,\overline Z_0)$ is
\[
 D\Phi_{(Z_0,\overline Z_0)}=(D\rho_0,\Theta_0),
\]
which is an isomorphism: its first component is an isomorphism on $E_1$,
and its second component is an isomorphism on $\ker D\rho_0$.  The
holomorphic inverse function theorem therefore shows that $\Phi$ is a
local biholomorphism.  Thus
\[
 (H,\Theta)\in M_r(\C)\times\C^{2rs-r^2},
 \qquad H=\rho(Z,\overline W),
\]
are local holomorphic coordinates near $(Z_0,\overline Z_0)$.

In these coordinates, write
\[
 \widehat G(H,\Theta):=G\bigl(\Phi^{-1}(H,\Theta)\bigr).
\]
Equation~\eqref{eq:polarized-boundary} says
$\widehat G(0,\Theta)=0$.  After shrinking the coordinate neighborhood,
we may assume that $(tH,\Theta)$ remains in it for $0\leq t\leq1$.
The fundamental theorem of calculus applied to
$t\mapsto\widehat G(tH,\Theta)$ gives
\[
 \widehat G(H,\Theta)
 =\int_0^1D_H\widehat G(tH,\Theta)[H]\,dt.
\]
For $K\in M_r(\C)$, define the complex-linear map
\begin{equation}\label{eq:explicit-M}
 \mathcal M_{Z,\overline W}(K)
 :=\int_0^1
 D_H\widehat G\bigl(t\rho(Z,\overline W),
                     \Theta(Z,\overline W)\bigr)[K]\,dt.
\end{equation}
The integrand is holomorphic in $(Z,\overline W)$ and complex-linear in
$K$, so the same is true of $\mathcal M_{Z,\overline W}$.  Substituting
$K=\rho(Z,\overline W)$ into \eqref{eq:explicit-M} yields
\begin{align*}
 \mathcal M_{Z,\overline W}\bigl(\rho(Z,\overline W)\bigr)
 &=\mathcal M_{Z,\overline W}\bigl(I_r-ZW^H\bigr)\\
 &=\int_0^1D_H\widehat G(tH,\Theta)[H]\,dt\\
 &=\widehat G(H,\Theta)\\
 &=G(Z,\overline W).
\end{align*}
Using the definitions of $\rho$ and $G$, this is precisely
\eqref{eq:polarized-factor}.  After shrinking $U_0$ if necessary, the
construction is valid for all $Z,W\in U_0$.
\end{proof}

Setting $W=Z$ in \eqref{eq:explicit-M} gives the explicit formula
\[
 \mathcal M_Z(K)
 =\int_0^1D_H\widehat G\bigl(t(I_r-ZZ^H),
                  \Theta(Z,\overline Z)\bigr)[K]\,dt,
 \qquad K\in M_r(\C),
\]
where $\mathcal M_Z:=\mathcal M_{Z,\overline Z}$.  In particular,
\eqref{eq:polarized-factor} becomes
\begin{equation}\label{eq:defining-factor}
 \mathcal M_Z\bigl(I_r-ZZ^H\bigr)
 =I_{r'}-F(Z)F(Z)^H.
\end{equation}
The map $Z\mapsto\mathcal M_{Z,\overline Z}$ is real-analytic rather
than holomorphic.  We only need its continuity later along the curve
$\gamma(t)$, which gives
$\mathcal M_{\gamma(t)}\to\mathcal M_{Z_0}$ as
$\gamma(t)\to Z_0$.  When $r=r'=1$,
\eqref{eq:defining-factor} is the usual scalar multiplier identity for a
proper map.

For any real direction $\dot Z$,
the product rule applied to \eqref{eq:defining-factor} gives
\[
 d(I_{r'}-FF^H)_{Z_0}(\dot Z)
 =(d\mathcal M)_{Z_0}(\dot Z)\bigl(\rho(Z_0)\bigr)
 +\mathcal M_{Z_0}\bigl(d\rho_{Z_0}(\dot Z)\bigr).
\]
The first term on the right is zero because $\rho(Z_0)=0$.  Therefore
\[
 d(I_{r'}-FF^H)_{Z_0}
 =\mathcal M_{Z_0}\circ d(I_r-ZZ^H)_{Z_0}.
\]
Since $d\rho_{Z_0}$ is onto the Hermitian matrices, this identity
determines $\mathcal M_{Z_0}$ on that real vector space.  Complex linearity
then determines it on all of $M_r(\C)$.

Because $F$ maps one boundary to the other, $dF_{Z_0}$ maps the complex
tangent space of $S_{r,s}$ into that of $S_{r',s'}$.  Under our
normalizations, there is therefore a linear map
\[
 L_{Z_0}:M_{r,s-r}(\C)\longrightarrow M_{r',s'-r'}(\C)
\]
such that
\[
 dF_{Z_0}(0,X)=(0,L_{Z_0}(X)).
\]

\begin{lemma}\label{lem:Levi}
For all $X,Y\in M_{r,s-r}(\C)$,
\begin{equation}\label{eq:Levi-natural}
 L_{Z_0}(X)L_{Z_0}(Y)^H=\mathcal M_{Z_0}(XY^H).
\end{equation}
\end{lemma}

\begin{proof}
Apply $-\partial\bar\partial$ to
\eqref{eq:defining-factor} in the directions $(0,X)$ and $(0,Y)$ at
$Z_0$.  Since
\[
 \rho(Z_0)=0,\qquad
 (\partial\rho)_{Z_0}(0,X)=0,\qquad
 (\bar\partial\rho)_{Z_0}(0,Y)=0,
\]
all terms in which a derivative falls on $\mathcal M$ vanish.  We are
left with
\[
 -\partial\bar\partial(I_{r'}-FF^H)_{Z_0}
       ((0,X),\overline{(0,Y)})
 =\mathcal M_{Z_0}\!\left(
   -\partial\bar\partial\rho_{Z_0}
       ((0,X),\overline{(0,Y)})\right).
\]
Because $F$ is holomorphic, no derivative of $\overline F$ occurs in a
$(1,0)$ direction and no derivative of $F$ occurs in a $(0,1)$
direction.  Thus the left side, computed exactly as in
\eqref{eq:source-Levi}, is
\[
 L_{Z_0}(X)L_{Z_0}(Y)^H.
\]
The right side is $\mathcal M_{Z_0}(XY^H)$ by
\eqref{eq:source-Levi}.  This proves the identity.
\end{proof}

\section{Positive points map to semi-positive points}
\label{sec:positive}

The global Shilov-boundary hypothesis controls the image of every
positive point, not only the image of a boundary point.  We use the
defining extremal property of the Shilov boundary to obtain the
following high-rank analogue of the corresponding interior statement
for a rank-one source.

\begin{theorem}\label{thm:positive-image}
If $F(S_{r,s})\subset S_{r',s'}$, then
$F(\Omega_{r,s})\subset\overline{\Omega}_{r',s'}$.
\end{theorem}

\begin{proof}
Let $a\in\C^{r'}$ and $b\in\C^{s'}$ be unit row vectors.  The scalar
function $aF(Z)b^H$ is holomorphic on a neighborhood of
$\overline{\Omega}_{r,s}$.  On $S_{r,s}$,
$F(Z)F(Z)^H=I_{r'}$, and therefore
$|aF(Z)b^H|\leq1$.  By the defining extremal property of the Shilov
boundary,
\[
 \sup_{Z\in\overline{\Omega}_{r,s}}|aF(Z)b^H|
 =\sup_{Z\in S_{r,s}}|aF(Z)b^H|\leq1.
\]
Thus the same estimate holds throughout $\Omega_{r,s}$.  For fixed $a$,
taking the supremum over all unit vectors $b$ gives $\|aF(Z)\|\leq1$.
Hence
\[
 a\bigl(I_{r'}-F(Z)F(Z)^H\bigr)a^H
 =1-\|aF(Z)\|^2\geq0.
\]
This holds for every unit row vector $a$, so
\begin{equation}\label{eq:target-defect-positive}
 I_{r'}-F(Z)F(Z)^H\succeq0
 \qquad(Z\in\Omega_{r,s}).
\end{equation}
Therefore
$F(Z)\in\overline{\Omega}_{r',s'}$.
\end{proof}

\section{A dimension estimate}\label{sec:linear}

To obtain the dimension estimate, we apply
\eqref{eq:Levi-natural} to tangent matrices having only one nonzero
column.  Different column positions will give mutually orthogonal
subspaces of $\C^{s'-r'}$, whose dimensions can then be added.  Let
$e_1,\ldots,e_r$ be the standard column basis of $\C^r$.  For a column
vector $x\in\C^r$
and $1\leq j\leq s-r$, let $X_j(x)$ be the matrix whose $j$-th
column is $x$ and whose other columns are zero, and set
\[
 L_j(x):=L_{Z_0}(X_j(x)).
\]
Since
\[
 X_i(x)X_j(y)^H=\delta_{ij}xy^H,
\]
Lemma~\ref{lem:Levi} gives
\begin{equation}\label{eq:orthogonal-L}
 L_i(x)L_j(y)^H=\delta_{ij}\mathcal M_{Z_0}(xy^H).
\end{equation}
Put $E_{\alpha\beta}=e_\alpha e_\beta^H$ and form the block matrix
\[
 \mathcal G_{Z_0}
 :=\bigl[\mathcal M_{Z_0}(E_{\alpha\beta})\bigr]_{\alpha,\beta=1}^r.
\]

\begin{lemma}\label{lem:capacity}
The matrix $\mathcal G_{Z_0}$ is positive semidefinite.  If
$g(Z_0):=\rankop\mathcal G_{Z_0}$, then there are
$A_1,\ldots,A_{g(Z_0)}\in M_{r',r}(\C)$ such that
\begin{equation}\label{eq:matrix-factorization}
 \mathcal M_{Z_0}(T)=\sum_{\mu=1}^{g(Z_0)}A_\mu T A_\mu^H
 \qquad(T\in M_r(\C)).
\end{equation}
Moreover,
\begin{equation}\label{eq:capacity}
 g(Z_0)(s-r)\leq s'-r'.
\end{equation}
\end{lemma}

\begin{proof}
For each column index $j=1,\ldots,s-r$, stack the $r$ matrices
$L_j(e_\alpha)$, where $\alpha=1,\ldots,r$, and write
\[
 \mathbf L_j=
 \begin{pmatrix}
  L_j(e_1)\\ \vdots\\ L_j(e_r)
 \end{pmatrix}
 \in M_{rr',\,s'-r'}(\C).
\]
The $(\alpha,\beta)$ block of
$\mathbf L_j\mathbf L_j^H$ is
$L_j(e_\alpha)L_j(e_\beta)^H$.  Equation
\eqref{eq:orthogonal-L} therefore gives
\[
 \mathcal G_{Z_0}=\mathbf L_j\mathbf L_j^H.
\]
Therefore, $\mathcal G_{Z_0}\succeq0$.  Since
$\rankop(BB^H)=\rankop B$ for every matrix $B$, we also have
$\rankop\mathbf L_j=g(Z_0)$ for every $j$.

A positive semidefinite matrix of rank $g(Z_0)$ is a sum of $g(Z_0)$
rank-one matrices, so write
\[
 \mathcal G_{Z_0}=\sum_{\mu=1}^{g(Z_0)}c_\mu c_\mu^H,
 \qquad
 c_\mu=
 \begin{pmatrix}a_{\mu1}\\ \vdots\\a_{\mu r}\end{pmatrix},
 \quad a_{\mu\alpha}\in\C^{r'}.
\]
Here $c_\mu$ and $a_{\mu\alpha}$ are column vectors.  Set
$A_\mu=[a_{\mu1}\ \cdots\ a_{\mu r}]$.  Comparing the
$(\alpha,\beta)$ blocks gives
\[
 \mathcal M_{Z_0}(E_{\alpha\beta})
 =\sum_{\mu=1}^{g(Z_0)}a_{\mu\alpha}a_{\mu\beta}^H
 =\sum_{\mu=1}^{g(Z_0)}A_\mu E_{\alpha\beta}A_\mu^H.
\]
Since $A_\mu$ has $r$ columns, $\rankop A_\mu\leq r$ for every $\mu$.
The matrix units $E_{\alpha\beta}$ form a basis of $M_r(\C)$, so this
proves \eqref{eq:matrix-factorization}.

It remains to count dimensions.  Define
\[
 \mathcal E_j
 :=\Span\{uL_j(x):x\in\C^r,\ u\in\C^{r'}\}
 \subset\C^{s'-r'}.
\]
Here $u$ is a row vector.  Because
$L_j(x)=\sum_\alpha x_\alpha L_j(e_\alpha)$, the vectors in
$\mathcal E_j$ are exactly the linear combinations of the rows of
$\mathbf L_j$.  Thus $\mathcal E_j$ is the span of the rows of
$\mathbf L_j$, and hence
\[
 \dim\mathcal E_j=\rankop\mathbf L_j=g(Z_0).
\]
If $i\ne j$, then for generators
$v=uL_i(x)$ and $w=u'L_j(y)$,
\[
 \langle v,w\rangle
 =uL_i(x)L_j(y)^H(u')^H=0
\]
by \eqref{eq:orthogonal-L}.  Thus the $s-r$ spaces $\mathcal E_j$ are
mutually orthogonal subspaces of $\C^{s'-r'}$.  Therefore their dimensions
add without overlap, and
\[
 (s-r)g(Z_0)
 =\sum_{j=1}^{s-r}\dim\mathcal E_j
 \leq\dim\C^{s'-r'}=s'-r'.
\]
This is \eqref{eq:capacity}.
\end{proof}

\section{From one boundary point to a fixed subspace}\label{sec:Hopf}

The purpose of this section is to pass from information at the single
point $Z_0$ to a subspace that works for every $Z$.  We first construct
a subspace from the left kernel of $\mathcal M_{Z_0}(I_r)$ and prove
that it is null with the required dimension.  We then combine
Theorem~\ref{thm:positive-image}, analytic discs, and the scalar Hopf
boundary lemma to prove that this same subspace is contained in
$V_{F(Z)}$ for every $Z$.

We use the following standard form of the Hopf boundary lemma; see
\cite[Lemma~3.4]{GT}.
If $u$ is $C^2$ in the unit disc, $C^1$ near $1$, satisfies
$u\geq0$, $\Delta u\leq0$, and $u(1)=0$, then either $u\equiv0$ or
\[
 \lim_{t\nearrow1}\frac{u(t)}{1-t}>0.
\]
We only need the following consequence: if the displayed limit is zero,
then $u$ vanishes identically.

Since $I_r-Z_0Z_0^H=0$, substituting $Z=W=Z_0$ into
\eqref{eq:polarized-factor} gives only $\mathcal M_{Z_0}(0)=0$.
Thus the boundary normalization imposes no condition on
$\mathcal M_{Z_0}(I_r)$; in particular, it need not equal $I_{r'}$ and
may have a nontrivial left kernel.  We use this left kernel in the
following lemma.

\begin{lemma}\label{lem:null-kernel}
Fix $Z_0\in S_{r,s}$ and set
\[
 N_{Z_0}:=\{v\in\C^{r'}:v\mathcal M_{Z_0}(I_r)=0\},
 \qquad
 K:=\Span\bigl\{[\,v,vF(Z_0)\,]:v\in N_{Z_0}\bigr\}
 \subset\C^{r'+s'}.
\]
Then $K$ is a null subspace and
\[
 \dim K=\dim N_{Z_0}
 =r'-\rankop\mathcal M_{Z_0}(I_r)
 \geq r'-g(Z_0)r.
\]
Moreover, every $v\in N_{Z_0}$ satisfies
\begin{equation}\label{eq:annihilate-normal}
 v\mathcal M_{Z_0}(T)=0
 \qquad(T\in M_r(\C)).
\end{equation}
\end{lemma}

\begin{proof}
Choose the matrices $A_\mu$ in
\eqref{eq:matrix-factorization}.  If $v\in N_{Z_0}$,
then
\[
 0=v\mathcal M_{Z_0}(I_r)v^H
 =\sum_{\mu=1}^{g(Z_0)}\|vA_\mu\|^2.
\]
Thus $vA_\mu=0$ for every $\mu$.  Therefore, for every
$T\in M_r(\C)$, \eqref{eq:matrix-factorization} gives
\[
 v\mathcal M_{Z_0}(T)
 =v\left(\sum_{\mu=1}^{g(Z_0)}A_\mu T A_\mu^H\right)
 =\sum_{\mu=1}^{g(Z_0)}(vA_\mu)T A_\mu^H
 =0.
\]

For $v,w\in N_{Z_0}$, the target Hermitian form gives
\[
 \big\langle[\,v,vF(Z_0)\,],
              [\,w,wF(Z_0)\,]\big\rangle_{r',s'}
 =vw^H-vF(Z_0)F(Z_0)^H w^H=0,
\]
where we used $F(Z_0)F(Z_0)^H=I_{r'}$.  By sesquilinearity, the Hermitian
form vanishes identically on $K\times K$, so $K$ is null.

Finally,
\[
 \rankop\mathcal M_{Z_0}(I_r)
 =\rankop\sum_{\mu=1}^{g(Z_0)}A_\mu A_\mu^H
 \leq\sum_{\mu=1}^{g(Z_0)}\rankop(A_\mu A_\mu^H)
 =\sum_{\mu=1}^{g(Z_0)}\rankop A_\mu
 \leq g(Z_0)r.
\]
The first inequality is the subadditivity of rank, and the equality uses
$\rankop(A_\mu A_\mu^H)=\rankop A_\mu$.  Since
$\mathcal M_{Z_0}(I_r)$ is positive semidefinite by
\eqref{eq:matrix-factorization},
\[
 \dim N_{Z_0}
 =r'-\rankop\mathcal M_{Z_0}(I_r)
 \geq r'-g(Z_0)r.
\]
The defining map of $K$ is injective, so $\dim K=\dim N_{Z_0}$.
\end{proof}

\begin{proposition}\label{prop:fixed-incidence}
With $K$ as in Lemma~\ref{lem:null-kernel},
\begin{equation}\label{eq:fixed-incidence}
 K\subset V_{F(Z)}\qquad(Z\in U).
\end{equation}
\end{proposition}

\begin{proof}

Let $Z\in\Omega_{r,s}$.  There is a fractional-linear automorphism
$\sigma\in\Aut(\Omega_{r,s})$ such that $\sigma(Z)=0$.
It and its inverse extend holomorphically past
$\overline{\Omega}_{r,s}$ and carry $S_{r,s}$ onto itself.  Put
$\widetilde Z_0=\sigma(Z_0)$.  Then
$\widetilde Z_0\widetilde Z_0^H=I_r$.  Consider the analytic disc
\[
 \gamma(\zeta):=\sigma^{-1}(\zeta\widetilde Z_0),\qquad |\zeta|<1,
\]
For $|\zeta|<1$,
\[
 I_r-(\zeta\widetilde Z_0)(\zeta\widetilde Z_0)^H
 =(1-|\zeta|^2)I_r>0
 \qquad(|\zeta|<1).
\]
Thus $\zeta\widetilde Z_0\in\Omega_{r,s}$, so $\gamma$ is well defined
and $\gamma(\Delta)\subset\Omega_{r,s}$.
Moreover, $\gamma(0)=Z$, $\gamma(1)=Z_0$, and $\gamma$ extends
holomorphically past $\zeta=1$.  We now evaluate the limit
\[
 D_\gamma:=\lim_{t\nearrow1}
 \frac{I_r-\gamma(t)\gamma(t)^H}{1-t}.
\]
To compute $D_\gamma$, let
\[
 \mathcal T=
 \begin{pmatrix}T_{11}&T_{12}\\T_{21}&T_{22}\end{pmatrix}
\]
be a matrix preserving $I_{r,s}$ which represents $\sigma^{-1}$.  The
four blocks have sizes $r\times r$, $r\times s$, $s\times r$, and
$s\times s$, respectively.  We have
\[
 [I_r\ W]\mathcal T
 =[T_{11}+WT_{21}\ \ T_{12}+WT_{22}].
\]
The first block is invertible, and normalizing it to $I_r$ gives
\[
 \sigma^{-1}(W)
 =(T_{11}+WT_{21})^{-1}(T_{12}+WT_{22}).
\]
Multiplying $\mathcal T I_{r,s}\mathcal T^H=I_{r,s}$ on the left by
$[I_r\ W]$ and on the right by its adjoint gives
\[
 I_r-WW^H
 =(T_{11}+WT_{21})
 \bigl(I_r-\sigma^{-1}(W)\sigma^{-1}(W)^H\bigr)
 (T_{11}+WT_{21})^H.
\]
Multiplying by the inverses of the two outer factors gives
\[
 I_r-\sigma^{-1}(W)\sigma^{-1}(W)^H
 =R(W)(I_r-WW^H)R(W)^H,
 \qquad R(W):=(T_{11}+WT_{21})^{-1}.
\]
Taking $W=t\widetilde Z_0$ in the preceding formula and using
$\widetilde Z_0\widetilde Z_0^H=I_r$, we obtain
\[
 I_r-\gamma(t)\gamma(t)^H
 =(1-t^2)R(t\widetilde Z_0)R(t\widetilde Z_0)^H.
\]
Therefore
\[
 \frac{I_r-\gamma(t)\gamma(t)^H}{1-t}
 =(1+t)R(t\widetilde Z_0)R(t\widetilde Z_0)^H.
\]
Letting $t\nearrow1$ gives
\begin{equation}\label{eq:disc-defect}
 D_\gamma=2R(\widetilde Z_0)R(\widetilde Z_0)^H>0.
\end{equation}
Here the last matrix is positive definite because
$R(\widetilde Z_0)$ is invertible.  Near $t=1$, apply
\eqref{eq:defining-factor} to the real curve
$t\mapsto\gamma(t)$.  By the complex linearity of
$\mathcal M_{\gamma(t)}$,
\[
 \frac{I_{r'}-F(\gamma(t))F(\gamma(t))^H}{1-t}
 =\mathcal M_{\gamma(t)}\left(
   \frac{I_r-\gamma(t)\gamma(t)^H}{1-t}\right).
\]
Since $\mathcal M_{\gamma(t)}\to\mathcal M_{Z_0}$, taking the limit and
using \eqref{eq:disc-defect} gives
\begin{equation}\label{eq:target-normal-derivative}
 \lim_{t\nearrow1}
 \frac{I_{r'}-F(\gamma(t))F(\gamma(t))^H}{1-t}
 =\mathcal M_{Z_0}(D_\gamma).
\end{equation}

For $v\in N_{Z_0}$, define
\[
 u_v(\zeta)
 :=v\bigl(I_{r'}-F(\gamma(\zeta))
              F(\gamma(\zeta))^H\bigr)v^H.
\]
Theorem~\ref{thm:positive-image} gives $u_v\geq0$.  Also $u_v(1)=0$ because
$\gamma(1)=Z_0$ and $F(Z_0)F(Z_0)^H=I_{r'}$.  If
$h_v(\zeta):=vF(\gamma(\zeta))$, then
\[
 u_v(\zeta)=\|v\|^2-\|h_v(\zeta)\|^2.
\]
Since $h_v$ is holomorphic,
\[
 \Delta u_v(\zeta)
 =-4\|h_v'(\zeta)\|^2
 \leq0.
\]
Equations \eqref{eq:annihilate-normal},
\eqref{eq:disc-defect}, and \eqref{eq:target-normal-derivative} show
that
\[
 \lim_{t\nearrow1}\frac{u_v(t)}{1-t}
 =v\mathcal M_{Z_0}(D_\gamma)v^H=0.
\]
Since $\gamma$ and $F\circ\gamma$ extend holomorphically past
$\zeta=1$, the function $u_v$ is $C^2$ in $\Delta$ and $C^1$ near
$\zeta=1$.  Moreover,
\[
 u_v\geq0,\qquad \Delta u_v\leq0,\qquad u_v(1)=0,
\]
and the preceding computation gives a zero radial difference quotient
at $1$.  Thus all the hypotheses of the Hopf boundary lemma stated
above are satisfied.  If $u_v$ were not identically zero, that lemma
would give a strictly positive limit, contrary to the preceding
computation.  Hence $u_v\equiv0$.

Since $u_v\equiv0$, the vector $h_v$ has constant norm.  The identity
$\Delta\|h_v\|^2=4\|h_v'\|^2$ then gives $h_v'\equiv0$, so $h_v$ is
constant.  Evaluating it at $0$ and $1$ yields
\[
 vF(Z)=vF(Z_0).
\]
The point $Z\in\Omega_{r,s}$ was arbitrary, and the holomorphic identity
principle extends this equality to the connected neighborhood $U$:
each entry of $vF(Z)-vF(Z_0)$ is holomorphic on $U$ and vanishes on the
open set $\Omega_{r,s}$.
Therefore,
\[
 [\,v,vF(Z_0)\,]=v[\,I_{r'},F(Z)\,]\in V_{F(Z)}
\qquad(Z\in U),
\]
which proves \eqref{eq:fixed-incidence}.
\end{proof}

\begin{remark}\label{rem:global-vs-local}
The inequality \eqref{eq:target-defect-positive} is the only place where
we use the condition on the whole Shilov boundary.  It makes the
function $u_v$ nonnegative on the entire disc used in
Proposition~\ref{prop:fixed-incidence}; the Hopf boundary lemma then
shows that the vectors coming from the kernel at $Z_0$ remain in
$V_{F(Z)}$ for every $Z$.

If $F$ is defined only near an open piece of the Shilov boundary, that
disc need not remain in the neighborhood, so
\eqref{eq:target-defect-positive} is unavailable on the whole disc.
Thus the first-derivative argument here does not by itself give a local
higher-gap theorem.  In the first gap, Kim--Zaitsev obtained a stronger
local result by also studying higher derivatives \cite{KZ13}.  A related
use of the Hopf lemma appears in Kim's work at general Shilov points
\cite{KimDuke}.  For $r=1$, the argument above gives, under the stronger
global hypothesis, an alternative proof of the conclusion in
\cite{GaoGap}.
\end{remark}

\section{Proof of the gap theorem}\label{sec:proof}

\begin{proof}[Proof of Theorem~\ref{thm:main}]
Fix $Z_0\in S_{r,s}$ and use the unitary normalizations from
Section~\ref{sec:Levi}, so that
$Z_0=(I_r,0)$ and $F(Z_0)=(I_{r'},0)$.  Lemma~\ref{lem:capacity} and the strict upper
inequality in \eqref{eq:gap} give
\[
 g(Z_0)(s-r)\leq s'-r'<(k+1)(s-r).
\]
Since $g(Z_0)$ is an integer, $g(Z_0)\leq k$.
Lemma~\ref{lem:null-kernel} and
Proposition~\ref{prop:fixed-incidence} therefore produce a fixed
null space
\[
 K\subset V_{F(Z)}\quad(Z\in U),
 \qquad \dim K\geq r'-kr.
\]

Set $d=r'-kr$ and choose a $d$-dimensional subspace
$K_0\subset K$.  The normalization $F(Z_0)=(I_{r'},0)$ gives
\[
 V_{F(Z_0)}=\{[x,x,0]:x\in\C^{r'}\}.
\]
Hence $K_0=\{[x,x,0]:x\in E\}$ for a $d$-dimensional subspace
$E\subset\C^{r'}$.  Choose a unitary matrix on $\C^{r'}$ carrying
$E$ to the span of the first $d$ coordinate vectors, and apply it
simultaneously to the positive coordinates and the first $r'$ negative
coordinates.  These target coordinate changes preserve
$\Omega_{r',s'}$ and $S_{r',s'}$, and we may assume that $E$ is spanned
by the first $d$ coordinate vectors.

For $v\in E$, Proposition~\ref{prop:fixed-incidence} gives
\[
 [v,v,0]\in K_0\subset V_{F(Z)}
 =\{[u,uF(Z)]:u\in\C^{r'}\}.
\]
Comparison of the first $r'$ coordinates forces $u=v$, and hence
\[
 vF(Z)=(v,0)\qquad(Z\in U).
\]
Thus, after splitting the first $d$ rows and columns, we can write
\[
 F(Z)=
 \begin{pmatrix}
  I_d&0\\
  A(Z)&B(Z)
 \end{pmatrix}.
\]
For $Z\in\Omega_{r,s}$, Theorem~\ref{thm:positive-image} now gives
\[
 0\preceq I_{r'}-F(Z)F(Z)^H
 =\begin{pmatrix}
    0&-A(Z)^H\\
    -A(Z)&I_{r'-d}-A(Z)A(Z)^H-B(Z)B(Z)^H
  \end{pmatrix}.
\]
A positive semidefinite Hermitian block matrix with a zero diagonal
block has zero corresponding off-diagonal block.  Consequently,
$A(Z)=0$ on $\Omega_{r,s}$, and the holomorphic identity principle gives
$A\equiv0$ on $U$.  Setting $F_0:=B$, we obtain
\[
 F(Z)=
 \begin{pmatrix}
  I_d&0\\
  0&F_0(Z)
 \end{pmatrix}.
\]
This is \eqref{eq:main-block}.
The remaining block has $r'-d=kr$ rows and
$s'-d=kr+s'-r'$ columns.  On $S_{r,s}$, the identity
$F(Z)F(Z)^H=I_{r'}$ gives $F_0(Z)F_0(Z)^H=I_{kr}$.  On
$\Omega_{r,s}$, Theorem~\ref{thm:positive-image} gives
$I_{kr}-F_0(Z)F_0(Z)^H\succeq0$.

Finally, if the source normalization above is $Z\mapsto ZQ$, then the
argument applies to $\widetilde F(Z)=F(ZQ^H)$.  Replacing $Z$ by $ZQ$
absorbs this normalization into the remaining block
$F_0(Z)=\widetilde F_0(ZQ)$.  Hence only the target coordinate changes
appear in the statement of the theorem.
\end{proof}

\begin{corollary}
Under the numerical assumptions of Theorem~\ref{thm:main}, no such map
can satisfy
\[
 F(\Omega_{r,s})\subset\Omega_{r',s'}.
\]
\end{corollary}

\begin{proof}
Theorem~\ref{thm:main} gives a nonzero null vector in every $V_{F(Z)}$
because $r'-kr>0$.  But the form is positive definite on $V_{F(Z)}$
when $F(Z)\in\Omega_{r',s'}$, a contradiction.
\end{proof}

\end{document}